\documentclass[11pt]{amsart}

\usepackage{amsmath,amsthm,amssymb}
\usepackage[margin=1in]{geometry}
\usepackage[bookmarks]{hyperref}
\usepackage[all,cmtip]{xy}
\usepackage{verbatim}
\usepackage{mathtools}

\newtheorem{theorem}{Theorem}[subsection]
\newtheorem{proposition}[theorem]{Proposition}
\newtheorem{conjecture}[theorem]{Conjecture}
\newtheorem{corollary}[theorem]{Corollary}
\newtheorem{lemma}[theorem]{Lemma}

\theoremstyle{definition}

\newtheorem{definition}[theorem]{Definition}

\numberwithin{equation}{subsection}

\newcommand{\Ext}{\operatorname{Ext}}

\newcommand{\Hom}{\operatorname{Hom}}

\newcommand{\St}{\operatorname{St}}

\begin{document}
%\LARGE

\title[Good $(p,r)$-Filtrations]
{\bf On good $\bf{(p,r)}$-filtrations for rational $G$-modules}

\begin{abstract} In this paper we investigate Donkin's $(p,r)$-Filtration Conjecture, and present two proofs of the ``if" direction of the statement when $p\geq 2h-2$. One 
proof involves the investigation of when the tensor product between the Steinberg module and a simple module has a good filtration. One of our main results 
shows that this holds under suitable requirements on the highest weight of the simple module. The second proof involves recasting Donkin's Conjecture in terms of 
the identifications of projective indecomposable $G_{r}$-modules with certain tilting $G$-modules, and establishing necessary cohomological criteria for the 
$(p,r)$-filtration conjecture to hold. 
\end{abstract}

\author{\sc Tobias Kildetoft}
\address 
{Department of Mathematics\\ Aarhus University\\
DK- 8000 Aarhus C, Denmark}
\thanks{Research of the first author was supported by QGM (Centre for Quantum Geometry of Moduli Spaces) funded by the Danish National Research Foundation}
\email{tobias.kildetoft@gmail.com}

\author{\sc Daniel K. Nakano}
\address
{Department of Mathematics\\ University of Georgia \\
Athens\\ GA~30602, USA}
\thanks{Research of the second author was supported in part by
NSF grant DMS-1002135}
\email{nakano@math.uga.edu}

\date\today
\thanks{2010 {\em Mathematics Subject Classification.} Primary
20J06;
Secondary 20G10}
%\subjclass{Primary 20J06; Secondary 20G10}
\maketitle
\section{Introduction}

\subsection{} Let $G$ be a simple, simply connected algebraic group scheme over the algebraically closed field $k$ of characteristic $p > 0$. It is well-known that the category of rational $G$-modules is not semisimple.  
Thus, one of the major open problems is to determine multiplicities of composition factors in modules which naturally arise from characteristic zero through reduction modulo $p$. The modules of interest are the 
induced modules  $\nabla(\lambda) = \rm{Ind}_B^G(\lambda)$ where $\lambda\in X_{+}$ (dominant integral weight) and $B$ is a Borel subgroup (corresponding to the negative roots). The characters of $\nabla(\lambda)$ are given by Weyl's character formula, 
and $\nabla(\lambda)$ has a simple socle $L(\lambda)$ where each finite-dimensional simple $G$-module is isomorphic to a unique such $L(\lambda)$. 

The modules $\nabla(\lambda)$ form the building blocks for studying injective modules, and it is natural to consider modules which admit filtrations whose sections are $\nabla(\lambda)$ for suitable $\lambda\in X_+$. These 
filtrations are called good filtrations. For each $\lambda\in X_+$ with unique decomposition $\lambda = \lambda_0 + p^r\lambda_1$ with $\lambda_0\in X_r$ ($p^{r}$th restricted weights) and $\lambda_1\in X_+$, one 
can define $\nabla^{(p,r)}(\lambda) = L(\lambda_0)\otimes \nabla(\lambda_1)^{(r)}$ where $(r)$ denotes the twisting of the module action by the $r$th Frobenius morphism. A $G$-module $M$ has 
a good $(p,r)$-filtration if and only if $M$ has a filtration with factors of the form $\nabla^{(p,r)}(\lambda)$ for suitable $\lambda\in X_+$. Let $\St_r = L((p^r-1)\rho)$ (which is also isomorphic to $\nabla((p^r-1)\rho)$) be the $r$th Steinberg module. 
The following conjecture which was introduced by Donkin at an MSRI lecture in 1990 interrelates good filtrations with good $(p,r)$-filtrations via the Steinberg module. 

\begin{conjecture} \label{donkinconj} Let $M$ be a finite-dimensional $G$-module. Then $M$ has a good $(p,r)$-filtration if and only if $\St_r\otimes M$ has a good filtration.
\end{conjecture} 

When $p\geq 2h-2$ (where $h$ is the Coxeter number of the root system $R$), Andersen \cite{andersen01} showed that if $M$ has a good $(p,r)$-filtration then  $\St_r\otimes M$ has a good filtration. 
The verification of the other direction of the conjecture appears to be much harder. A special case is that for any $\lambda\in X_+$, the module $\nabla(\lambda)$ has a good $p$-filtration. This is a special case since 
tensor products of modules with good filtrations again have good filtrations. Parshall and Scott \cite{parshallscottpfilt} have proved that $\nabla(\lambda)$ has a good $p$-filtration when the characteristic is large enough that the Lusztig character 
formula holds.

\subsection{}In this paper, we will primarily focus on issues related to the ``if'' direction of Conjecture~\ref{donkinconj}. First, we will expand on the results of Andersen by proving that when $M$ has a good $(p,r)$-filtration then $\St_r\otimes M$ has a good filtration, 
provided a suitable inequality holds between $p$, $r$, $h$ and the weights occurring in the good $(p,r)$-filtration of $M$. As a special case, we recover the results of Andersen, though our method of proof is different. 
Our method of proof involves the use of Donkin's cohomological criterion for the existence of a good filtration, and a careful analysis of the vanishing of extension groups with  suitable conditions on weights of the form 
$w.0 + p\beta$ with $w\in W$ and $\beta\in \mathbb{Z}R$. 

In order to prove the ``if'' direction of  Conjecture~\ref{donkinconj}, it is clearly enough to prove that $\St_r\otimes \nabla^{(p,r)}(\lambda)$ has a good filtration for any $\lambda\in X_+$. However, due to a result of Andersen 
(included as \ref{reductiontosimple}), it turns out that the ``if'' direction is equivalent to $\St_r\otimes L(\lambda)$ having a good filtration for any $\lambda\in X_r$. The inequality we obtain allows us to prove that $\St_r\otimes L(\lambda)$ has a good filtration 
with smaller restrictions on $p$ provided that the weight $\lambda$ is also suitably smaller. This still leaves weights $\lambda\in X_+$ for which we do not know whether $\St_r\otimes \nabla^{(p,r)}(\lambda)$ has a good filtration when $p$ is small. 
However, if $\lambda = \lambda_0 + p^r\lambda_1$ with $\lambda_0\in X_r$ and if $\lambda_1$ is large enough compared to $\lambda_0$ (made precise in \ref{storlambda1}), then we can still show that $\St_r\otimes \nabla^{(p,r)}(\lambda)$ 
has a good filtration, even if $\lambda_0$ is not small enough to satisfy the inequality we get with respect to $p$, $r$ and $h$.

A natural question is for which $\lambda\in X_+$ does $\St_r\otimes L(\lambda)$ have a good filtration? When $p\geq 2h-2$ and if $\langle\lambda,\alpha_0^{\vee}\rangle \leq (p^r-1)(h-1)$ (where $\alpha_0$ is the highest short root of $R$) we 
have that $L(\lambda)\simeq \nabla^{(p,r)}(\lambda)$ (\ref{nablapsimple}) so in these cases it does hold. However, we also prove that this is close to being the best bound of this type possible. Namely, we show that if $p = 2h - 5$ and $R$ is of type $A$, 
then there is a $\lambda$ with $\langle\lambda,\alpha_0^{\vee}\rangle \leq (p - 1)(h-1)$ and such that $\St_1\otimes L(\lambda)$ does not have a good filtration. Furthermore, we demonstrate that our results are strong enough to prove the ``if'' direction of the $(p,r)$-filtration conjecture for root system of type $A_{2}$, $A_{3}$, and $B_{2}$ over fields of 
arbitrary characteristic, as well as for the root system of type $G_2$ as long as $p\neq 7$. The methods used to analyze the root systems of small rank, together with the ``negative'' examples mentioned above, provide a deeper insight into why the condition on $\lambda$ in order for $\St_r \otimes L(\lambda)$ to have a good filtration should be $\lambda\in X_r$ rather than a bound on $\langle\lambda,\alpha_0^{\vee}\rangle$.

In the final section of the paper, we recast Donkin's $(p,r)$-Filtration Conjecture via tilting modules. This allows us to establish a cohomological criterion (analogous to the one for good filtrations) for $\text{St}_{r}\otimes M$ to admit  
a good filtration (see Theorem~\ref{T:cohom-criteria}). This cohomological criterion is independent of the characteristic of the field, and has many similarities to those developed in \cite{andersen01}, though we focus on a different set of modules, which allows us to get a vanishing criterion only involving $\Ext^1$-groups, rather than having to involve all higher $\Ext$-groups. As a corollary of this result we show that if Donkin's Tilting Module Conjecture holds then the ``if'' direction of Donkin's $(p,r)$-Filtration Conjecture holds. Since the tilting module conjecture is valid when $p\geq 2h-2$, this yields a second proof of the ``if'' direction of the $(p,r)$-filtration conjecture. We note that if both directions of the $(p,r)$-conjecture are true then our 
cohomological criteria is equivalent to a module $M$-admitting a good $(p,r)$-filtration.

\subsection{Acknowledgements} The first author acknowledges the financial support of the Danish National Research Foundation and the hospitality of the Department of Mathematics at the University of Georgia during Spring 2013. 
The work in this paper was initiated during this period of time. 

\section{Preliminaries} 

\subsection{Notation}
\noindent 
Throughout this paper, the following basic notation will be used. 
\begin{itemize}

\item $k$: an algebraically closed field of characteristic $p> 0$.

\item $G$: a simple, simply connected algebraic group scheme over $k$, defined over $\mathbb{F}_p$ (the assumption of $G$ being simple is for convenience and the results easily generalize to $G$ reductive).

\item $T$: a maximal split torus in $G$. 

\item $R$: the corresponding (irreducible) root system associated to $(G,T)$. When referring to short and long roots, when a root system has roots of only one length, all roots shall be considered as both short and long.

\item $R^{\pm}$: the positive (respectively, negative) roots.  

\item $S = \{\alpha_1,\alpha_2,\dots,\alpha_n\}$: an ordering of the simple roots.

\item  $B$: a Borel subgroup containing $T$ corresponding to the negative roots. 

\item $\mathbb E$: the Euclidean space spanned by $\Phi$ with inner product $\langle\,,\,\rangle$ normalized 
so that $\langle\alpha,\alpha\rangle=2$ for $\alpha \in \Phi$ any short root.

\item $X=X(T)=\mathbb Z \omega_1\oplus\cdots\oplus{\mathbb Z}\omega_n$: the weight lattice, where the 
fundamental dominant weights $\omega_i\in{\mathbb E}$ are defined by $\langle\omega_i,\alpha_j^\vee\rangle=\delta_{ij}$, $1\leq i,j\leq n$.

\item $X_{+}=X(T)_{+}={\mathbb N}\omega_1+\cdots+{\mathbb N}\omega_n$: the dominant weights.

\item $X_{r}=X_{r}(T)=\{\lambda\in X(T)_+: 0\leq \langle\lambda,\alpha^\vee\rangle<p^{r},\,\,\forall \alpha\in S \}$: the set of $p^{r}$-restricted dominant weights.

\item $F:G\rightarrow G$: the Frobenius morphism. 

\item $G_r=\text{ker }F^{r}$: the $r$th Frobenius kernel of $G$. 

\item $W$: the Weyl group of $R$. 

\item $w_{0}$: the long element of the Weyl group.

\item $\alpha^\vee=2\alpha/\langle\alpha,\alpha\rangle$: the coroot of $\alpha\in R$.

\item $\rho$: the Weyl weight defined by $\rho=\frac{1}{2}\sum_{\alpha\in\Phi^+}\alpha$.

\item $h$: the Coxeter number of $\Phi$, given by $h=\langle\rho,\alpha_0^{\vee} \rangle+1$.

\item $\alpha_0$: the maximal short root.

\item $\leq$ on $X(T)$: a partial ordering of weights, for $\lambda, \mu \in X(T)$, $\mu\leq \lambda$ if and only if $\lambda-\mu$ is a linear combination 
of simple roots with non-negative integral coefficients. 

\item $M^{(r)}$:  the module obtained by composing the underlying representation for 
a rational $G$-module $M$ with $F^{r}$.

\item $\nabla(\lambda) := \operatorname{Ind}_B^G\lambda$, $\lambda\in X(T)_{+}$: the induced module whose character is provided by Weyl's character formula.  

\item $\Delta(\lambda)$, $\lambda\in X(T)_{+}$: the Weyl module of highest weight $\lambda$. Thus, $\Delta(\lambda)\cong \nabla(-w_{0}(\lambda))^*$.

\item $L(\lambda)$: the simple finite dimensional $G$-module with highest weight $\lambda\in X(T)_{+}$. 

\item $\nabla^{(p,r)}(\lambda) = L(\lambda_0)\otimes \nabla(\lambda_1)^{(r)}$: where $\lambda = \lambda_0 + p^r\lambda_1$ with $\lambda_0\in X_r$ and $\lambda_1\in X_+$.

\item $\nabla^p(\lambda) = \nabla^{(p,1)}(\lambda)$.

\end{itemize} 

\subsection{Weights of the form $w\cdot 0$}

Throughout this section we will make use of the observation that if $\lambda\leq \mu$ with $\lambda, \mu\in X$ then $\langle\lambda,\alpha_0^{\vee}\rangle \leq \langle\mu,\alpha_0^{\vee}\rangle$.
This follows because $\mu = \lambda + \beta$ where $\beta$ is a non-negative linear combination of the simple roots, and the inner product of any simple root with $\alpha_{0}^{\vee}$ is greater than
or equal to zero. The ``dot'' action is given by $w\cdot \lambda = w(\lambda + \rho)-\rho$ for $w\in W$ and $\lambda\in X$. Define 
$R_w^{< 0} = \{\alpha\in -R^+\mid w^{-1}(\alpha) > 0\}$. We begin by stating a well-known fact \cite[Lemma 7.3.6]{goodmanwallach} relating $w\cdot 0$ and $R_w^{< 0}$. Note that the second part of the lemma 
follows since $w_0\cdot 0 = w_0(\rho) - \rho = -2\rho = \sum_{\alpha\in -R^+}\alpha$. 

\begin{lemma}\label{negativ}
Let $w\in W$ and set $R_w^{< 0} = \{\alpha\in -R^+\mid w^{-1}(\alpha) > 0\}$. Then $$w\cdot 0 = \sum_{\alpha\in R_w^{<0}}\alpha$$
In particular,  $w_0\cdot 0\leq w\cdot 0\leq 0$.
\end{lemma}
We can now prove bounds on the size of the inner products of $w\cdot 0$ with coroots.

\begin{proposition}\label{innerproductalpha0} Let $w\in W$. Then 
\begin{itemize} 
\item[(a)] $\langle w\cdot 0,\alpha_0^{\vee}\rangle \geq -2(h-1)$;
\item[(b)] $\langle w\cdot 0,\alpha^{\vee}\rangle\leq h-2$ for any $\alpha\in S$. 
\end{itemize} 
\end{proposition}

\begin{proof} (a) By Lemma \ref{negativ} we have $w\cdot 0\geq w_0\cdot 0$, thus  
$$\langle w\cdot 0,\alpha_0^{\vee}\rangle \geq \langle w_0\cdot 0,\alpha_0^{\vee}\rangle = \langle -2\rho,\alpha_0^{\vee}\rangle = -2(h-1).$$ 

(b) Observe that $\langle w\cdot 0,\alpha^{\vee}\rangle = \langle w(\rho) - \rho,\alpha^{\vee}\rangle = \langle w(\rho),\alpha^{\vee}\rangle - \langle\rho,\alpha^{\vee}\rangle = \langle\rho,w^{-1}(\alpha^{\vee})\rangle - 1$. 
But, $\langle\rho,w^{-1}(\alpha^{\vee})\rangle$ is at most $\langle\rho,\alpha_0^{\vee}\rangle$ because for any root $\beta$ we have that $\langle\rho,\beta^{\vee}\rangle$ is the height of $\beta^{\vee}$. 
Part (b) now follows because $\langle\rho,\alpha_0^{\vee}\rangle = h-1$. 
\end{proof}

\subsection{Dominant weights in the root lattice} We summarize the results on dominant weights which will be used in the subsequent sections in the following proposition. 

\begin{proposition}\label{summary} Let $\lambda\in X_{+}$ and assume that $p\geq h-1$.  
\begin{itemize} 
\item[(a)] If $\lambda = w\cdot 0 + p\beta$ for some $w\in W$ and $\beta\in\mathbb{Z}R$, 
then $\langle\lambda,\alpha_0^{\vee}\rangle \geq 2(p - h + 1)$.
\item[(b)] Suppose that $\Ext^{n}_G(k,L(\lambda))\neq 0$ for some $n\geq 0$.
Then $\langle\lambda,\alpha_0^{\vee}\rangle \geq 2(p - h + 1)$.
\end{itemize} 
\end{proposition}

\begin{proof}
(a) By Proposition \ref{innerproductalpha0}(b), if $\alpha\in S$, we have $\langle w\cdot 0,\alpha^{\vee}\rangle \leq h-2$. Since $\lambda$ is dominant 
$$0\leq \langle w\cdot 0+p\beta, \alpha^{\vee} \rangle \leq (h-1)+p\langle \beta,\alpha^{\vee} \rangle.$$ 
Now by assumption $p\geq h-1$ which forces $\langle \beta,\alpha^{\vee} \rangle\geq 0$, so $\beta$ must be dominant. 

Next we observe that $\langle\beta,\alpha_0^{\vee}\rangle \geq 2$ because $\beta\neq 0$ and if $\langle \beta,\alpha_{0}^{\vee} \rangle =1$ then $\beta$ must be a 
minuscule weight which can be viewed as a non-zero class in $X/{\mathbb Z}R$ (see \cite[Lemma II.12.10]{rags}).  This would contradict the fact that $\beta$ is in the root lattice. Combining this fact 
with Proposition \ref{innerproductalpha0}(a), we see that $\langle\lambda,\alpha_0^{\vee}\rangle = \langle w\cdot 0,\alpha_0^{\vee}\rangle + p\langle\beta,\alpha_0^{\vee}\rangle \geq -2(h-1) + 2p = 2(p - h + 1)$ as claimed.

(b) The linkage principle (\cite[Corollary II.6.17]{rags}) implies that $\lambda = w\cdot 0 + p\beta$ for some $w\in W$ and some $\beta\in \mathbb{Z}R$, so the result follows directly from part (a).  
\end{proof}

Note that when $p\geq h$ the above cannot be improved. This is because, for $\lambda = (p-h +1)\alpha_0$, there is a short exact sequence $0\to L(\lambda)\to \nabla(\lambda)\to L(0)\to 0$, as can be seen by applying the Jantzen sum formula (\cite[Proposition 8.19]{rags}).

\section{Filtrations} 

\subsection{} Let $M$ be a rational $G$-module. In this paper a {\em $G$-filtration} for $M$ is an increasing sequence of $G$-submodules of $M$: $M_{0}\subseteq M_{1} \subseteq \dots \subseteq M$ such that $\cup_{i}M_{i}=M$. 
We now present the definition of good and {good $(p,r)$-filtration}. 

\begin{definition} Let $M$ be a $G$-module 
\begin{itemize}
\item[(a)] $M$ has a {\em good filtration} if and only if it has a $G$-filtration such that for each $i$, $M_{i+1}/M_{i}\cong \nabla(\lambda_{i})$ where $\lambda_{i}\in X_+$. 
\item[(b)] $M$ has a {\em good $(p,r)$-filtration} if and only if it has a $G$-filtration such that for each $i$, $M_{i+1}/M_{i}\cong \nabla^{(p,r)}(\lambda_{i})$ where $\lambda_{i}\in X_+$.
\item[(c)] $M$ has a {\em good $p$-filtration} if and only if it has a good $(p,1)$-filtration.
\end{itemize}
\end{definition} 

\subsection{Good Filtrations} The following well-known result will be the main tool used to prove the existence of good filtrations.

\begin{theorem}[{\cite[Corollary 1.3]{donkin81}},{\cite[Proposition II.4.16]{rags}}]\label{cohcrit} Let $M$ be a $G$-module. The following are equivalent 
\begin{itemize} 
\item[(a)] $M$ has a good filtration. 
\item[(b)] $\Ext_G^1(\Delta(\mu),M) = 0$ for all $\mu\in X_+$.
\item[(c)] $\Ext_G^n(\Delta(\mu),M) = 0$ for all $\mu\in X_+$, $n\geq 1$.
\end{itemize}
\end{theorem}

For our purposes it is convenient to provide a modified version of this cohomological criterion. We note that if Donkin's conjecture holds, parts (b) and (c) of the theorem below would give 
a cohomological criterion for the existence of good $(p,r)$-filtrations. 

\begin{theorem} \label{altconditions} Let $M$ be a $G$-module. The following are equivalent 
\begin{itemize} 
\item[(a)] $\St_{r}\otimes M$ has a good filtration. 
\item[(b)] $\Ext_{G/G_r}^1(k,\Hom_{G_r}(\Delta(\mu),\St_r\otimes M))=0$ for all $\mu\in X_+$.
\item[(c)] $\Ext_{G/G_r}^n(k,\Hom_{G_r}(\Delta(\mu),\St_r\otimes M))=0$ for all $\mu\in X_+$, $n\geq 1$.
\end{itemize} 
\end{theorem} 

\begin{proof} Consider the Lyndon-Hochschild-Serre spectral sequence 
$$E_{2}^{i,j}=\text{Ext}^{i}_{G/G_{r}}(k,\text{Ext}^{j}_{G_{r}}(\Delta(\mu),\St_r\otimes M))\Rightarrow \text{Ext}^{i+j}_{G}(\Delta(\mu),\St_r \otimes M).$$ 
Since $\St_r$ is injective as a $G_r$-module by \cite[Proposition II.10.2]{rags}, we also have that $\St_r\otimes M$ is injective as a $G_r$-module. Therefore, this 
spectral sequence collapses and yields the isomorphism: 
\begin{equation} 
\text{Ext}^{n}_{G}(\Delta(\mu),\St_r \otimes M)\cong \Ext_{G/G_r}^n(k,\Hom_{G_r}(\Delta(\mu),\St_r\otimes M))
\end{equation} 
for all $n\geq 0$. The theorem now follows from Theorem~\ref{cohcrit}. 
\end{proof} 

In the paper we will also employ the following important property of good filtrations.

\begin{theorem}[{\cite[Theorem 1]{mathieu90}},{\cite[Proposition II.4.21]{rags}}]\label{tensorproductgoodfilt}
If $M$ and $M'$ are $G$-modules, each of which has a good filtration, then $M\otimes M'$ also has a good filtration.
\end{theorem}

\section{Good filtrations for $\St_{r}\otimes L(\lambda)$: bounds on $\lambda$}

\subsection{} From the proof of Theorem~\ref{altconditions}, we have 
\begin{equation} \label{ext1iso} 
\text{Ext}^{1}_{G}(\Delta(\mu),\St_r \otimes M)\cong \Ext_{G/G_r}^1(k,\Hom_{G_r}(\Delta(\mu),\St_r\otimes M))
\end{equation} 
We will first show that we can restrict our attention to a finite set of weights $\mu\in X_{+}$ in order to 
verify that this extension group is zero. 

\begin{lemma}\label{smallmu} Let $\lambda,\mu\in X_+$ and $\Ext_G^1(\Delta(\mu),\St_r\otimes L(\lambda))\neq 0$.
Then $$\langle\mu,\alpha_0^{\vee}\rangle\leq \langle\lambda,\alpha_0^{\vee}\rangle + (p^r - 1)(h-1).$$
\end{lemma}

\begin{proof}
Consider the short exact sequence $0\to L(\lambda)\to \nabla(\lambda)\to Q\to 0$. 
One can tensor this sequence with $\St_r$ and apply $\Hom_G(\Delta(\mu),-)$ to obtain the long exact sequence 
$$\cdots\to\Hom_G(\Delta(\mu),\St_r\otimes Q)\to\Ext_G^1(\Delta(\mu),\St_r\otimes L(\lambda))\to\Ext_G^1(\Delta(\mu),\St_r\otimes\nabla(\lambda))\to\cdots$$
Since $\St_r\otimes\nabla(\lambda)$ has a good filtration by Theorem \ref{tensorproductgoodfilt}, it follows from Theorem \ref{cohcrit} that $\Ext_G^1(\Delta(\mu),\St_r\otimes\nabla(\lambda)) = 0$. 
From our hypothesis, $\Ext_G^1(\Delta(\mu),\St_r\otimes L(\lambda))\neq 0$ which implies that $\Hom_G(\Delta(\mu),\St_r\otimes Q)\neq 0$. 

The head of $\Delta(\mu)$ is $L(\mu)$, so $\mu$ must then be a weight of $\St_r\otimes Q$ and also a weight of $\St_r\otimes \nabla(\lambda)$.
In particular, $\mu\leq (p^r - 1)\rho + \lambda$, and 
$$\langle\mu,\alpha_0^{\vee}\rangle \leq \langle(p^r - 1)\rho + \lambda,\alpha_0^{\vee}\rangle = (p^r - 1)(h-1) + \langle\lambda,\alpha_0^{\vee}\rangle.$$ 
\end{proof}

\subsection{} Using the result in the preceding section we can then obtain another bound for $\mu$ which is needed in order to get $\Ext_G^1(\Delta(\mu),\St_r\otimes L(\lambda))\neq 0$, this time requiring $\langle\mu,\alpha_0^{\vee}\rangle$ to be large enough compared to $\lambda$, $p$ and $r$. 

\begin{proposition}\label{largemu} Let $p\geq h-1$ and $\Ext_G^1(\Delta(\mu),\St_r\otimes L(\lambda))\neq 0$ for some $\lambda,\mu\in X_+$. Then 
$$\langle\mu,\alpha_0^{\vee}\rangle \geq (p^r - 1)(h-1) + 2p^r(p - h + 1) - \langle\lambda,\alpha_0^{\vee}\rangle.$$
\end{proposition}

\begin{proof} From (\ref{ext1iso}) and using the fact that $\St_r\cong \St_r^{*}$ and $\Delta(\mu)^{*} \cong \nabla(-w_0(\mu))$ we have 
\begin{eqnarray*} 
\Ext_G^1(\Delta(\mu),\St_r\otimes L(\lambda))&\cong& \Ext_{G/G_r}^1(k,\Hom_{G_r}(\Delta(\mu),\St_r\otimes L(\lambda)))\\
&\cong& \Ext_{G/G_r}^1(k,\Hom_{G_r}(\St_r,\nabla(-w_0(\mu))\otimes L(\lambda)))
\end{eqnarray*}
so one can assume the last Ext-group is not $0$. Set $\nu = -w_0(\mu)$.

If we take a composition series for $\nabla(\nu)\otimes L(\lambda)$ (as a $G$-module), this gives us a filtration of $\Hom_{G_r}(\St_r,\nabla(\nu)\otimes L(\lambda))$ since $\St_r$ is projective as a $G_r$-module. 
Therefore, since we assume that $\Ext_{G/G_r}^1(k,\Hom(\St_r,\nabla(\nu)\otimes L(\lambda)))\neq 0$, there must be some $\sigma\in X_+$ such that $L(\sigma)$ is a composition factor of $\nabla(\nu)\otimes L(\lambda)$ and such that $\Ext_{G/G_r}^1(k,\Hom_{G_r}(\St_r,L(\sigma)))\neq 0$.

In particular, $\Hom_{G_r}(\St_r,L(\sigma))\neq 0$. By Steinberg's Tensor Product Theorem (STPT) \cite[Corollary II.3.17]{rags}, $L(\sigma)\cong L(\sigma_0)\otimes L(\sigma_1)^{(r)}$ 
where $\sigma = \sigma_0+p^r\sigma_1$ with $\sigma_0\in X_r$ and $\sigma_1\in X_+$. Consequently, 
$$0\neq \Hom_{G_r}(\St_r,L(\sigma)) \cong \Hom_{G_r}(\St_r,L(\sigma_0)\otimes L(\sigma_1)^{(r)})\cong \Hom_{G_r}(\St_r,L(\sigma_0))\otimes L(\sigma_1)^{( r )}.$$
Since $\St_r$ is simple as a $G_r$-module, $\sigma_0 = (p^r - 1)\rho$, and $\Hom_{G_r}(\St_r,L(\sigma))\cong L(\sigma_1)^{(r)}$. One now has 
$$0\neq \Ext_{G/G_r}^1(k,\Hom_{G_r}(\St_r,L(\sigma)))\cong \Ext_G^1(k,L(\sigma_1)).$$ 

Now apply Proposition~\ref{summary}(b), which shows that $\langle\sigma_1,\alpha_0^{\vee}\rangle\geq 2(p - h + 1)$. This yields 
$$\langle\sigma,\alpha_0^{\vee}\rangle \geq (p^r - 1)(h-1) + 2p^r(p - h + 1).$$

Since $L(\sigma)$ was assumed to be a composition factor of $\nabla(\nu)\otimes L(\lambda)$ we get that $\sigma\leq \nu + \lambda$. 
Therefore, 
$$(p^r - 1)(h-1) + 2p^r(p - h + 1) \leq \langle\sigma,\alpha_0^{\vee}\rangle \leq \langle\nu+\lambda,\alpha_0^{\vee}\rangle.$$
It now follows that 
$$\langle \mu,\alpha_{0}^{\vee}\rangle= \langle 
\mu, -w_{0}(\alpha_{0}^{\vee}) \rangle =\langle\nu,\alpha_0^{\vee}\rangle \geq (p^r - 1)(h-1) + 2p^r(p - h + 1) - \langle\lambda,\alpha_0^{\vee}\rangle.
$$
\end{proof}

\subsection{} The results in the preceding sections allow us provide sufficient conditions for $\St_r\otimes L(\lambda)$ to admit a  good filtration.

\begin{theorem}\label{mumellem} Let $p\geq h-1$ and $\Ext^{1}_G(\Delta(\mu),\St_r\otimes L(\lambda))\neq 0$ for some $\lambda,\mu\in X_+$. Then 
$$(p^r - 1)(h-1) + 2p^r(p - h + 1) - \langle\lambda,\alpha_0^{\vee}\rangle \leq \langle\mu,\alpha_0^{\vee}\rangle\leq \langle\lambda,\alpha_0^{\vee}\rangle + (p^r - 1)(h-1)$$
\end{theorem}

\begin{proof}
This follows directly by combining Lemma \ref{smallmu} and Proposition \ref{largemu}.
\end{proof}

And we also obtain the following theorem, removing the mention of $\mu$. We only state the result for $p\geq h$ as the conditions on $\lambda$ are never satisfied for $p = h-1$.

\begin{theorem}\label{lambdaalpha0}
Assume that $p\geq h$ and let $\lambda\in X_+$ with $\langle\lambda,\alpha_0^{\vee}\rangle < p^r(p - h + 1)$.
Then $\St_r\otimes L(\lambda)$ has a good filtration.
\end{theorem}

\begin{proof}
If $\St_r\otimes L(\lambda)$ does not have a good filtration, then by Theorem \ref{cohcrit} there must be some $\mu\in X_+$ with $\Ext_G^1(\Delta(\mu),\St_r\otimes L(\lambda))\neq 0$. Hence, 
by Theorem~\ref{mumellem} 
$$(p^r - 1)(h-1) + 2p^r(p - h + 1) - \langle\lambda,\alpha_0^{\vee}\rangle \leq \langle\mu,\alpha_0^{\vee}\rangle\leq \langle\lambda,\alpha_0^{\vee}\rangle + (p^r - 1)(h-1)$$ 
and in particular $$(p^r - 1)(h-1) + 2p^r(p - h + 1) - \langle\lambda,\alpha_0^{\vee}\rangle \leq \langle\lambda,\alpha_0^{\vee}\rangle + (p^r - 1)(h-1)$$ which gives $\langle\lambda,\alpha_0^{\vee}\rangle \geq p^r(p - h + 1)$, contradicting the choice of $\lambda$.
\end{proof}

\subsection{} In this section, we will present another method to obtain sufficient conditions on $p$, $r$ and $\lambda$ to ensure that $\St_r\otimes L(\lambda)$ has a good filtration.
This procedure starts with the $r=1$ case and then uses an inductive argument similar to the one in \cite[Proposition 2.10]{andersen01}, though the formulation below is more general.
In some cases, it will be easier to deal with the $r=1$ case. 

\begin{proposition}\label{induction}
Let $m$ be a positive integer and let $\Gamma_1\subseteq X_m$ be a set of weights, such that $\St_m\otimes L(\lambda)$ has a good filtration for all $\lambda\in \Gamma_1$. Let $\Gamma_r = \sum_{i=0}^{r-1}p^{im}\Gamma_1$ be the set of weights $\lambda$ of the form $\lambda = \lambda_0 + p^m\lambda_1+\cdots + p^{(r-1)m}\lambda_{r-1}$ with all $\lambda_i\in \Gamma_1$.

Then $\St_{rm}\otimes L(\lambda)$ has a good filtration for all $\lambda\in \Gamma_r$.
\end{proposition}

\begin{proof}
We proceed by induction on $r$. Using the STPT, we have 
$$\St_{rm} = L((p^{rm} - 1)\rho) = L((p^m - 1)\rho + p^m(p^{(r-1)m}-1)\rho) \cong \St_m\otimes \St_{(r-1)m}^{(m)}$$ 
and if $\lambda = \lambda_0 + p^m\lambda_1 + \cdots + p^{(r-1)m}\lambda_{r-1}$ with all $\lambda_i\in \Gamma_1$, then we can write $\lambda = \lambda_0 + p^m\mu$ where $\mu = \lambda_1 + 
p^m\lambda_2+\cdots +p^{(r-2)m}\lambda_{r-1}\in \Gamma_{r-1}$, and $L(\lambda)\cong L(\lambda_0)\otimes L(\mu)^{(m)}$.

Now $\St_{rm}\otimes L(\lambda) \cong \St_m\otimes L(\lambda_0) \otimes (\St_{(r-1)m}\otimes L(\mu))^{(m)}$.  By assumption $\St_m\otimes L(\lambda_0)$ has a good filtration since $\lambda_0\in \Gamma_1$, and by induction have that $\St_{(r-1)m}\otimes L(\mu)$ has a good filtration. The result follows by Proposition \ref{reductiontosimple}.
\end{proof}

The case of the above that will be of most interest is when $m=1$. As a special case one obtains the result: if $\St_1\otimes L(\lambda)$ has a good filtration for all $\lambda\in X_1$, then  
$\St_r\otimes L(\lambda)$ has a good filtration for all $\lambda\in X_r$.

One way to use Proposition~\ref{induction} is to use Theorem~\ref{lambdaalpha0} in the case $r=1$ to get a set of weights to use, and then expand. The set of weights thus obtained for arbitrary $r$ will generally contain weights not satisfying the inequality of 
Theorem \ref{lambdaalpha0} unless either $p \leq h$ or $p\geq 2h - 2$. Note that  if $p < 2h -2$ then there are weights in $X_1$ which do not satisfy the inequality in Theorem \ref{lambdaalpha0}, so we cannot directly improve the bound on $p$ this way.

%In the case of $p\geq 2h - 2$ we can take $\Gamma_1 = X_1$ and obtain an alternative proof of Corollary \ref{recover}. 

\section{Donkin's Conjecture}

\subsection{} We first recall the result of Andersen (cf. \cite[Proposition 2.6]{andersen01}), which allows us to reduce the general question of whether $\St_r\otimes \nabla^{(p,r)}(\lambda)$ has a good filtration to just considering the case when $\lambda\in X_r$ (i.e., 
whether $\St_r\otimes L(\lambda)$ has a good filtration). A proof is included below as we need a slightly more general version than the one originally given by Andersen.

\begin{proposition}\label{reductiontosimple} Let $V$ be a $G$-module. The following are equivalent.
\begin{itemize} 
\item[(a)] $\St_r\otimes V$ has a good filtration.
\item[(b)] $\St_r\otimes V\otimes\nabla(\lambda)^{(r)}$ has a good filtration for all $\lambda\in X_+$.
\end{itemize}
\end{proposition}

\begin{proof} Since $\nabla(0)^{(r)} \cong k^{(r)} \cong k$, we clearly have that (b) implies (a). For the other direction (i.e., (a) implies (b)) assume that $\St_r\otimes V$ has a good filtration. Then by Theorem \ref{tensorproductgoodfilt}, 
$\St_r\otimes V\otimes \nabla(p^r\lambda)$ has a good filtration for all $\lambda\in X_+$. Since direct summands of modules with good filtrations themselves have good filtrations (by Theorem \ref{cohcrit}), it is sufficient to show that $\St_r\otimes\nabla(\lambda)^{(r)}$ 
is a direct summand of $\St_r\otimes\nabla(p^r\lambda)$, which would then imply that  $\St_r\otimes V\otimes\nabla(\lambda)^{(r)}$ is a direct summand of $\St_r\otimes V\otimes\nabla(p^r\lambda)$.

For our purposes we need to show that there are maps $\varphi: \St_r\otimes \nabla(\lambda)^{(r)}\to \St_r\otimes \nabla(p^r\lambda)$ and $\psi: \St_r\otimes\nabla(p^r\lambda)\to \St_r\otimes \nabla(\lambda)^{(r)}$ 
such that $\psi\circ\varphi = \rm{id}$. Since $\St_r\otimes\nabla(\lambda)^{(r)}\cong \nabla((p^r-1)\rho + p^r\lambda)$ \cite[Proposition II.3.19]{rags}, and this has a simple socle and $1$-dimensional space of endomorphisms,
 it is sufficient to find such $\varphi$ and $\psi$ such that the weight space of weight $(p^r-1)\rho + p^r\lambda$ is not in the kernel of the composed map.

In order for maps $\varphi$ and $\psi$ as above to exist, we must have $\Hom_G(\St_r\otimes\nabla(\lambda)^{(r)},\St_r\otimes\nabla(p^r\lambda))\neq 0$ and $\Hom_G(\St_r\otimes\nabla(p^r\lambda),\St_r\otimes\nabla(\lambda)^{(r)})\neq 0$. We prove this below, and note that the arguments for this claim also show that choosing any non-zero maps $\varphi$ and $\psi$ will in fact give the desired property.

By Frobenius reciprocity, we have $$\Hom_G(\nabla(\lambda)^{(r)},\nabla(p^r\lambda))\cong\Hom_B(\nabla(\lambda)^{(r)},p^r\lambda)\neq 0$$ since $p^r\lambda$ is the highest weight of $\nabla(\lambda)^{(r)}$. 
Hence, $\Hom_G(\St_r\otimes\nabla(\lambda)^{(r)},\St_r\otimes\nabla(p^r\lambda))\neq 0$.

On the other hand, we have $\St_r\otimes\nabla(\lambda)^{(r)}\cong\nabla((p^r - 1)\rho + p^r\lambda)$, and by Frobenius reciprocity, 
\begin{eqnarray*} 
\Hom_G(\St_r\otimes\nabla(p^r\lambda),\St_r\otimes\nabla(\lambda)^{(r)}) &\cong & \Hom_G(\St_r\otimes\nabla(p^r\lambda),\nabla((p^r-1)\rho + p^r\lambda)) \\
&\cong &\Hom_B(\St_r\otimes\nabla(p^r\lambda),(p^r-1)\rho + p^r\lambda)\\
&\neq& 0
\end{eqnarray*} 
since $(p^r-1)\rho + p^r\lambda$ is the highest weight of $\St_r\otimes\nabla(p^r\lambda)$. This shows that $\St_r\otimes\nabla(\lambda)^{(r)}$ is a direct summand of $\St_r\otimes\nabla(p^r\lambda)$ which completes the proof.
\end{proof}

%Theorem \ref{cohort} means that in order to show that $\St_r\otimes L(\lambda)$ has a good filtration, we need to show that for any $\mu\in X_+$, we have $\Ext_G^1(\Delta(\mu),\St_r\otimes L(\lambda)) = 0$.

\subsection{} We can now show that the ``if'' direction of Donkin's Conjecture is equivalent to proving that $\St_1\otimes L(\lambda)$ has a good filtration for all $\lambda\in X_1$. 

\begin{theorem} \label{ifDonkinequivalent}  The following statements are equivalent. 
\begin{itemize} 
\item[(a)] $\operatorname{St}_{1}\otimes L(\lambda)$ has a good filtration for all $\lambda\in X_{1}$. 
\item[(b)] $\operatorname{St}_{r}\otimes L(\lambda)$ has a good filtration for all $\lambda\in X_{r}$. 
\item[(c)] If $M$ is rational $G$-module such that $M$ has a good $(p,r)$-filtration then $\operatorname{St}_{r}\otimes M$ has a good filtration. 

\end{itemize}
\end{theorem} 

\begin{proof} $(a)\iff (b)$: This follows by using Proposition~\ref{induction}. $(c)\Rightarrow (b)$: This holds because $L(\lambda)$ has a good $(p,r)$-filtration. 

$(b)\Rightarrow (c)$: Suppose that $M$ has a good $(p,r)$-filtration. One can look at the sections and it suffices to 
show $\text{St}_{r}\otimes L(\lambda_{0})\otimes \nabla(\lambda_{1})^{(r)}$ has a good filtration for $\lambda_{0}\in X_{r}$, $\lambda_{1}\in X_{+}$. This follows by 
Theorem~\ref{reductiontosimple} because $\text{St}_{r}\otimes L(\lambda_{0})$ has a good filtration by hypothesis. 

\end{proof}

\subsection{} We can now present a proof of one direction of Donkin's Conjecture when $p\geq 2h-2$ which recovers Proposition 2.10 of \cite{andersen01}.

\begin{theorem}\label{recover} Let $p\geq 2h - 2$. If $M$ has a good $(p,r)$-filtration then $\St_r\otimes M$ has a good filtration.
\end{theorem}

\begin{proof} Since $p\geq 2h - 2$ we have $p- h+ 1 \geq h-1$. Therefore, for any $\lambda\in X_r$, 
$$\langle\lambda,\alpha_0^{\vee}\rangle\leq (p^r - 1)(h-1) < p^r(h-1) \leq p^r(p-h+1).$$ 
According to Theorem~\ref{lambdaalpha0}, $\St_r\otimes L(\lambda)$ has a good filtration for all $\lambda\in X_r$. The statement of the theorem now follows from Theorem~\ref{ifDonkinequivalent}.
\end{proof}

\subsection{} For dealing with specific cases, it will be useful to know for which dominant weights $\lambda$ we have $L(\lambda)\cong \nabla(\lambda)$ when $R$ is of type $A$. The following formulation is taken from the last part of II.8.21 in \cite{rags} with the difference that we have the requirement $\alpha-\beta_0\in R\cup\{0\}$ instead of $\alpha - \beta_0\in R$ (without this change the formulation is not correct).

\begin{theorem}[{\cite[Satz 9]{jantzen73}}]\label{simpleweights}
Assume that $R$ is of type $A_n$ and let $\lambda\in X_+$. For each $\alpha\in R^+$ write $\langle\lambda+\rho,\alpha^{\vee}\rangle = a_{\alpha}p^{s_{\alpha}} + b_{\alpha}p^{s_{\alpha}+1}$ for natural numbers $a_{\alpha},b_{\alpha},s_{\alpha}$ with $0 < a_{\alpha} < p$.

Then $L(\lambda)\cong\nabla(\lambda)$ if and only if for all $\alpha\in R^+$ there are positive roots $\beta_0,\beta_1,\dots,\beta_{b_{\alpha}}$ such that $\langle\lambda+\rho,\beta_0^{\vee}\rangle = a_{\alpha}p^{s_{\alpha}}$ and for all $1\leq i\leq b_{\alpha}$ we have $\langle\lambda+\rho,\beta_i^{\vee}\rangle = p^{s_{\alpha}+1}$ and such that further $\alpha = \sum_{i=0}^{b_{\alpha}}\beta_i$ and $\alpha-\beta_0\in R\cup\{0\}$.
\end{theorem}

Note that when applying the above theorem to determine whether $L(\lambda)\cong\nabla(\lambda)$ for some $\lambda\in X_+$, we only need to consider those $\alpha\in R^+$ with $\langle\lambda+\rho,\alpha^{\vee}\rangle > p$, since the condition is trivially satisfied for all other positive roots (by picking $\beta_0 = \alpha$ since in that case we have $b_{\alpha} = 0$).

\section{Tensoring with other simple modules}

\subsection{} The methods employed in this paper use the condition that $\langle\lambda,\alpha_0^{\vee}\rangle$ is not too large to show that $\St_r\otimes L(\lambda)$ has a good filtration. In particular, 
our techniques do not need that $L(\lambda)$ remains simple when restricted to $G_r$ (i.e.,  $\lambda\in X_r$).

Therefore, a natural question to ask is whether one can replace the conjecture that $\St_r\otimes L(\lambda)$ has a good filtration for all $\lambda\in X_r$ (which is still only a conjecture when $p < 2h - 2$) with the stronger statement that $\St_r\otimes L(\lambda)$ has a good filtration for all $\lambda\in X_+$ with $\langle\lambda,\alpha_0^{\vee}\rangle\leq (p^r - 1)(h-1)$ (which also holds when $p\geq 2h-2$ by Corollary \ref{lambdaalpha0}).

However, we will show that this is not the case for smaller primes in the following section. For $p\geq 2h-2$ with $\langle\lambda,\alpha_0^{\vee}\rangle\leq (p^r - 1)(h-1)$, 
$\lambda$ also satisfies $\langle\lambda,\alpha_0^{\vee}\rangle < p^r(p-h+1)$, and we have the following result which does hold for smaller primes. 

\begin{proposition}\label{nablapsimple}
Let $\lambda\in X_+$ with $\langle\lambda,\alpha_0^{\vee}\rangle < p^r(p-h+1)$. Then $L(\lambda)\cong \nabla^{(p,r)}(\lambda)$.
\end{proposition}

\begin{proof}
Write $\lambda = \lambda_0 + p^r\lambda_1$ with $\lambda_0\in X_r$ and observe that $$\langle\lambda_1,\alpha_0^{\vee}\rangle \leq \frac{1}{p^r}\langle\lambda,\alpha_0^{\vee}\rangle < p-h+1$$ so $\langle\lambda_1+\rho,\alpha_0^{\vee}\rangle < (p-h+1) + (h-1) = p$ and hence $L(\lambda_1)\cong\nabla(\lambda_1)$ by \cite[Corollary 5.6]{rags}.

Consequently, 
$$\nabla^{(p,r)}(\lambda) = L(\lambda_0)\otimes\nabla(\lambda_1)^{(r)} \cong L(\lambda_0)\otimes L(\lambda_1)^{(r)}\cong L(\lambda_0 + p^r\lambda_1) = L(\lambda).$$ 
\end{proof}

\subsection{} The following class of counterexamples shows that $\St_r\otimes L(\lambda)$ need not have a good filtration 
when $\langle\lambda,\alpha_0^{\vee}\rangle\leq (p^r - 1)(h-1)$. 

\begin{proposition}
Let $R$ be of type $A_n$ with $n\geq 3$ and assume that $p = 2h - 5$ is a prime. Let $\lambda = p(\omega_1+\omega_2+\cdots+\omega_{n-1})$. Then $\langle\lambda,\alpha_0^{\vee}\rangle \leq (p-1)(h-1)$ 
but $\St_1\otimes L(\lambda)$ does not have a good filtration.
\end{proposition}

\begin{proof}
Since $h\geq 4$ (recall that $h = n+1$) we have 
\begin{eqnarray*}
\langle\lambda,\alpha_0^{\vee}\rangle &=& p(n-1) = (2h - 5)(h-2) = 2h^2 - 9h + 10 = 2h^2 - 8h + 6 - (h-4)\\
&\leq& 2h^2 - 8h + 6 = (2h - 6)(h-1) = (p-1)(h-1)
\end{eqnarray*} 
which was the first part of the claim.

Let $\mu = \omega_1+\omega_2+\cdots+\omega_{n-1}$ (so $\lambda = p\mu$). We claim that $\nabla(\mu)\not\cong L(\mu)$. First apply Theorem \ref{simpleweights} with the positive root $\alpha = \alpha_1+\alpha_2+\cdots+\alpha_{n-1}$. One has  
$$\langle\mu+\rho,\alpha^{\vee}\rangle = (n-1) + (n-1) = 2h - 4 = 1 + p$$ 
so if we had $\nabla(\mu)\cong L(\mu)$ there would have to be a positive root $\beta_0$ with $\langle\mu+\rho,\beta_0^{\vee}\rangle = 1$ and $\alpha - \beta_0\in R\cup \{0\}$.

However, the only positive root $\beta_0$ such that $\langle\mu+\rho,\beta_0^{\vee}\rangle = 1$ is $\beta_0 = \alpha_n$ since $\mu+\rho$ is dominant and all other simple roots $\gamma$ have $\langle\mu+\rho,\gamma^{\vee}\rangle = 2$. Since 
$\alpha-\alpha_n$ is not a root (and $\alpha\neq\alpha_n$), this shows the claim.

Now one has  $\St_1\otimes L(\lambda) \cong \St_1\otimes L(\mu)^{(1)} \cong L((p-1)\rho + p\mu)$ by STPT. But since this is a simple module, the only way it can have a good filtration is if it is isomorphic to $\nabla((p-1)\rho + p\mu)$. But, 
by \cite[Proposition II.3.19]{rags}we have $\nabla((p-1)\rho + p\mu)\cong \St_1\otimes\nabla(\mu)^{(1)}$.  Since $L(\mu)^{(1)}$ is a submodule of $\nabla(\mu)^{(1)}$, it follows that if $\St_1\otimes L(\mu)^{(1)}\cong \St_1\otimes\nabla(\mu)^{(1)}$ then also $L(\mu)^{(1)}\cong \nabla(\mu)^{(1)}$, and thus $L(\mu)\cong \nabla(\mu)$, which is not the case.
\end{proof}

The existence of the aforementioned family of counterexamples means that if one wants to show that $\St_1\otimes L(\lambda)$ has a good filtration for all $\lambda\in X_1$ when $p$ is small, then the 
methods need to take into account more than just the size of $\langle\lambda,\alpha_0^{\vee}\rangle$.

\section{Good filtrations on  $\St_r\otimes \nabla^{(p,r)}(\lambda)$}

\subsection{} The goal of this section is to show that if $\lambda = \lambda_0 + p^r\lambda_1$ with $\lambda_0\in X_r$ and $\lambda_1\in X_+$ where $\lambda_1$ is ``large enough'' compared to $\lambda_0$ (and $p$ and $r$), then $\St_r\otimes \nabla^{(p,r)}(\lambda)$ has a good filtration, even if $\lambda_0$ is not small enough compared to $p$ and $r$ to 
apply Theorem~\ref{lambdaalpha0}. We start with some preliminary lemmas. The first lemma involves weights of $\nabla(\lambda)$ and the second provides conditions on 
when $R^1\rm{Ind}_B^G(\lambda) = 0$.

\begin{lemma}\label{weightsinsimple} Let $\lambda\in X_+$ and assume that $\mu$ is a weight of $\nabla(\lambda)$.
Then $\langle\mu,\alpha^{\vee}\rangle \geq -\langle\lambda,\alpha_0^{\vee}\rangle$ for all $\alpha\in R^+$. In particular, the same inequality holds for any weight of $L(\lambda)$.
\end{lemma}

\begin{proof} Since the Weyl group acts transitively on the Weyl chambers, there is some $w\in W$ such that $\langle w(\mu),\alpha^{\vee}\rangle\leq 0$ for all $\alpha\in R^+$. Then 
$\langle w(\mu),\alpha^{\vee}\rangle \geq \langle w(\mu),\alpha_0^{\vee}\rangle$ for all $\alpha\in R$ for the same reason that the reverse inequality would hold if $w(\mu)$ was dominant.

For any $\alpha\in R^+$, $\langle\mu,\alpha^{\vee}\rangle = \langle w(\mu),w(\alpha^{\vee})\rangle\geq \langle w(\mu),\alpha_0^{\vee}\rangle$. Since $\mu$ is a weight of $\nabla(\lambda)$, so is $w(\mu)$.  Hence, 
$w(\mu)\geq w_0(\lambda)$ and $\langle w(\mu),\alpha_0^{\vee}\rangle \geq \langle w_0(\lambda),\alpha_0^{\vee}\rangle$. Furthermore, $\langle w_0(\lambda),\alpha_0^{\vee}\rangle = -\langle\lambda,\alpha_0^{\vee}\rangle$ which gives the first claim.
If $\mu$ is a weight of $L(\lambda)$ then $\mu$ is also a weight of $\nabla(\lambda)$, thus the second claim follows.
\end{proof}

\begin{lemma}\label{indvanish}
Let $\lambda\in X$. If $\langle\lambda,\alpha^{\vee}\rangle \geq -1$ for all $\alpha\in S$ then $R^1\rm{Ind}_B^G(\lambda) = 0$.
\end{lemma}

\begin{proof}
This follows by combining \cite[Proposition II.4.5]{rags} and \cite[Proposition II.5.4(a)]{rags}.
\end{proof}

\subsection{} The lemmas in the preceding section can be used to show when $\nabla(\nu)\otimes L(\lambda)$ has a good filtration. 

\begin{proposition}\label{bignu}
Let $\lambda,\nu\in X_+$ with $\langle\lambda,\alpha_0^{\vee}\rangle \leq \langle\nu,\alpha^{\vee}\rangle + 1$ for all $\alpha\in S$.
Then $\nabla(\nu)\otimes L(\lambda)$ has a good filtration.
\end{proposition}

\begin{proof}
First note that by Lemma \ref{weightsinsimple}, for any weight $\mu$ of $L(\lambda)$ and any $\alpha\in S$, 
$$\langle\nu+\mu,\alpha^{\vee}\rangle = \langle\nu,\alpha^{\vee}\rangle + \langle\mu,\alpha^{\vee}\rangle \geq \langle\nu,\alpha^{\vee}\rangle - \langle\lambda,\alpha_0^{\vee}\rangle \geq -1.$$
Now apply the tensor identity (\cite[Proposition 3.6]{rags}) which gives 
$$\nabla(\nu)\otimes L(\lambda) = \rm{Ind}_B^G(\nu)\otimes L(\lambda) \cong \rm{Ind}_B^G(\nu \otimes L(\lambda)).$$ 
The weights of $L(\lambda)$ gives a filtration of $L(\lambda)$ as a $B$-module, so we obtain a filtration of $\nu\otimes L(\lambda)$ with factors of the form $\nu + \mu$ where 
$\mu$ is a weight of $L(\lambda)$.

We wish to show that this filtration gives a filtration of $\rm{Ind}_B^G(\nu\otimes L(\lambda))$ with terms of the form $\rm{Ind}_B^G(\nu + \mu)$. In order to do this, it is sufficient to show that 
$R^1\rm{Ind}_B^G(\nu + \mu) = 0$ for all weights $\mu$ of $L(\lambda)$. But, for any such $\mu$ and $\alpha\in S$ one has
 $\langle\nu+\mu,\alpha^{\vee}\rangle \geq -1$, so this follows by Lemma \ref{indvanish}. Therefore, we have demonstrated that $\nabla(\nu)\otimes L(\lambda)$ has a filtration with factors of the form 
 $\rm{Ind}_B^G(\gamma)$ for suitable $\gamma$ which finishes the proof.
\end{proof}

As a direct consequence of the above, we get a sufficient condition on $\lambda$ which guarantees that $\St_r\otimes L(\lambda)$ has a good filtration, with no requirement on $p$. For $p=h$ this condition is better than the one obtained from Theorem \ref{lambdaalpha0}.

\begin{theorem}\label{smallp}
If $\lambda\in X_+$ with $\langle\lambda,\alpha_0^{\vee}\rangle\leq p^r$ then $\St_r\otimes L(\lambda)$ has a good filtration.
\end{theorem}

\begin{proof}
This follows directly from Proposition \ref{bignu} since $\langle(p^r-1)\rho,\alpha^{\vee}\rangle = p^r - 1$ for all $\alpha\in S$.
\end{proof}

\subsection{} We now present sufficient conditions to insure that $\St_r\otimes\nabla^{(p,r)}(\lambda)$ has a good filtration. 

\begin{theorem}\label{storlambda1}
Let $\lambda$ be a dominant weight and write $\lambda = \lambda_0 + p^r\lambda_1$ with $\lambda_0\in X_r$. Moreover, assume that 
$\langle\lambda_0,\alpha_0^{\vee}\rangle \leq p^r(\langle\lambda_1,\alpha^{\vee}\rangle + 1)$ for all $\alpha\in S$. Then $\St_r\otimes\nabla^{(p,r)}(\lambda)$ has a good filtration.
\end{theorem}

\begin{proof}
By \cite[Proposition II.3.19]{rags}
$$\St_r\otimes \nabla^{(p,r)}(\lambda) = \St_r\otimes\nabla(\lambda_1)^{(r)}\otimes L(\lambda_0) \cong \nabla((p^r - 1)\rho + p^r\lambda_1)\otimes L(\lambda_0)$$ so the claim follows from Proposition \ref{bignu} since $\langle (p^r-1)\rho + p^r\lambda_1,\alpha^{\vee}\rangle = p^r\langle\lambda_1,\alpha^{\vee}\rangle + p^r-1$ for any $\alpha\in S$.
\end{proof}

As a special case of above theorem, we see that if $p^r(\langle\lambda_1,\alpha^{\vee}\rangle + 1) \geq (p^r-1)(h-1)$ for all $\alpha\in S$, then for any $\lambda = \lambda_0 + p^r\lambda_1$ with $\lambda_0\in X_r$, 
$\St_r \otimes \nabla^{(p,r)}(\lambda)$ has a good filtration.

\section{Root systems of small rank}

\subsection{} For the root systems of type $A_2$, $A_3$, $B_2$ and $G_2$ we can show that $\St_r\otimes M$ has a good filtration for any $G$-module $M$ with a good $(p,r)$-filtration, without any restrictions on $p$, except for the case $p=7$ in type $G_2$.
This will be accomplished by proving that $\St_1\otimes L(\lambda)$ has a good filtration for all $\lambda\in X_1$. The statement will then follow from Proposition \ref{induction} and Proposition \ref{reductiontosimple}.

In the following, we will call a weight $\lambda\in X_+$ {\em simple} if $L(\lambda) = \nabla(\lambda)$. We start with a result similar to Lemma \ref{smallmu} and Proposition \ref{largemu}. We will not give a proof here, as the arguments 
are completely identical to those of the mentioned results.

\begin{proposition}\label{smallrank} Let $\lambda\in X_1$ and assume that $\St_1\otimes L(\lambda)$ does not have a good filtration.
Then there are weights $\mu,\sigma\in X_+$ with $\mu\neq\lambda$ such that 
\begin{itemize}
\item[(a)] $\Ext_G^1(k,L(\sigma))\neq 0$, 
\item[(b)] $[\nabla(\lambda):L(\mu)]\neq 0$ and 
\item[(c)] $p\sigma \leq \lambda + \mu$.
\end{itemize} 
In particular, $\lambda$ is not simple, $p\sigma \leq 2\lambda$, $p\langle\sigma,\alpha_0^{\vee}\rangle \leq 2\langle\lambda,\alpha_0^{\vee}\rangle$, and $p\langle\sigma,\alpha_0^{\vee}\rangle \leq \langle\lambda+\mu,\alpha_0^{\vee}\rangle$.
\end{proposition}

In order to apply the above, we start by obtaining a version of Proposition \ref{summary}(b) when $p\leq h-1$ (for $p\geq h$ we can use the proposition itself, and as remarked there, we cannot improve this). We do this by using that if $\Ext_G^1(k,L(\sigma))\neq 0$ then $[\nabla(\sigma):L(0)]\neq 0$ (\cite[Proposition II.2.14]{rags}), and then apply the Jantzen sum formula (\cite[Proposition II.8.19]{rags}) to see which $\sigma$ satisfies this. Note that in type $G_2$, we instead use the tables in \cite{speciale} to do this. In some cases, we will explicitly determine those small weights $\sigma$ for which $\Ext_G^1(k,L(\sigma))\neq 0$, since this will be helpful, as can be seen from Proposition \ref{smallrank}.

Once we have obtained the above, we apply Proposition \ref{smallrank} in several steps. The first step is to use it to reduce the set of weights we need to consider. In some cases, we will instead apply Theorem \ref{smallp} for this. 
In some cases, we will also need to apply the Jantzen sum formula to $\nabla(\lambda)$ for some of those weights $\lambda$ we are left with, (or use the tables in \cite{speciale}), in order to know which simple modules can occur 
as composition factors of $\nabla(\lambda)$. This will also give a further reduction in the weights we need to consider, as we do not need to consider any simple weights.

In the following we will write all weights in the basis consisting of the fundamental weights.

\subsection{Type $\mathbf{A_2}$} Since we have $2h-2 = 4$, we need to consider the cases $p=2$ and $p=3$. 

\subsubsection{$\mathbf{p=2}$} In this case we are done as soon as we apply Theorem \ref{smallp} as there are no weights left to consider.

\subsubsection{$\mathbf{p=3}$} The only weight left to consider after applying Theorem \ref{smallp} is $(2,2) = (p-1)\rho$ which is simple, so we are done.

\subsection{Type $\mathbf{A_3}$} Since $2h-2 = 6$, the cases we need to consider are $p=2$, $p=3$ and $p=5$.

\subsubsection{$\mathbf{p=2}$} The only weight left to consider after applying Theorem \ref{smallp} is $(1,1,1) = (p-1)\rho$, which is simple so we are done.

\subsubsection{$\mathbf{p=3}$} We see that if $\Ext_G^1(k,L(\sigma))\neq 0$ then $\langle\sigma,\alpha_0^{\vee}\rangle \geq 2$ since all the fundamental weights are simple. The weight $(0,2,0)$ shows that we cannot do any better, but this is the only such weight where equality holds (as can be checked using the Jantzen sum formula).

By Theorem \ref{smallp}, we need to consider the weights $(0,2,2)$, $(1,1,2)$, $(1,2,1)$, $(1,2,2)$, $(2,0,2)$, $(2,1,1)$, $(2,1,2)$, $(2,2,0)$, $(2,2,1)$ and $(2,2,2)$. But applying Theorem \ref{simpleweights} we see that the only ones of these we need to consider are $(1,1,2)$, $(1,2,1)$, $(2,0,2)$, $(2,1,1)$ and $(2,1,2)$ as the rest are simple.

By Proposition \ref{smallrank} we can then further restrict to those weights $\lambda$ such that either $3(0,2,0) \leq 2\lambda$ or $\langle\lambda,\alpha_0^{\vee}\rangle \geq 5$. This rules out the weights $(1,1,2)$, $(2,0,2)$ and $(2,1,1)$, so we are left with just $(1,2,1)$ and $(2,1,2)$.

Applying the Jantzen sum formula to $\nabla(1,2,1)$ we see that we have a short exact sequence 
$$0\to L(1,2,1)\to\nabla(1,2,1)\to L(0,2,0)\to 0,$$ 
and we can thus apply Proposition \ref{smallrank} to rule out this weight, as we do not have $3(0,2,0)\leq (1,2,1) + (0,2,0)$ since $(1,2,1)+(0,2,0)-3(0,2,0) = (1,-2,1) = -\alpha_2$.

For the weight $(2,1,2)$ we again apply the Jantzen sum formula and get a short exact sequence 
$$0\to L(2,1,2)\to\nabla(2,1,2)\to L(0,1,0)\to 0.$$ 
Like before, we can rule out this weight as we do not have $3(0,2,0)\leq (0,1,0)+(2,1,2)$ since $(0,1,0)+(2,1,2) - 3(0,2,0) = (2,-4,2) = -2\alpha_2$.

\subsubsection{$p=5$} After applying Theorem \ref{lambdaalpha0}, we are left with the weights $(2,4,4)$, $(3,3,4)$, $(3,4,3)$, $(3,4,4)$, $(4,2,4)$, $(4,3,3)$, $(4,3,4)$, $(4,4,2)$, $(4,4,3)$ and $(4,4,4)$. But applying Theorem \ref{simpleweights} we reduce this to just the weights $(3,3,4)$, $(3,4,3)$, $(4,2,4)$, $(4,3,3)$ and $(4,3,4)$. And the result for $(3,3,4)$ follows from the result for $(4,3,3) = -w_0(3,3,4)$.

By Proposition \ref{summary}, if $\Ext_G^1(k,L(\sigma))\neq 0$ then $\langle\sigma,\alpha_0^{\vee}\rangle \geq 4$. Thus, by Proposition \ref{smallrank} we see that it will be sufficient, for each weight $\lambda$ in the above list, to show that if $L(\mu)$ is a composition factor of $\nabla(\lambda)$ with $\mu\neq \lambda$ then $\langle\mu + \lambda,\alpha_0^{\vee}\rangle < 5\cdot 4 = 32$, ie that $\langle\mu,\alpha_0^{\vee}\rangle \leq 19 - \langle\lambda,\alpha_0^{\vee}\rangle$.

Applying the Jantzen sum formula, we get the following (we do not need to compute the characters completely, as we only need bounds on the weights occuring):
\begin{itemize}
\item[(i)] For $\lambda = (4,3,3)$, all weights $\mu$ that occur have $\langle\mu,\alpha_0^{\vee}\rangle \leq 7 < 19 - \langle\lambda,\alpha_0^{\vee}\rangle = 9$.
\item[(ii)] For $\lambda = (4,3,4)$, all weights $\mu$ that occur have $\langle\mu,\alpha_0^{\vee}\rangle \leq 3 < 19 - \langle\lambda,\alpha_0^{\vee}\rangle = 8$.
\item[(iii)] For $\lambda = (4,2,4)$, all weights $\mu$ that occur have $\langle\mu,\alpha_0^{\vee}\rangle \leq 4 < 19 - \langle\lambda,\alpha_0^{\vee}\rangle = 9$.
\item[(iv)] For $\lambda = (3,4,3)$, all weights $\mu$ that occur have $\langle\mu,\alpha_0^{\vee}\rangle \leq 6 < 19 - \langle\lambda,\alpha_0^{\vee}\rangle = 9$.
\end{itemize} 
Thus we have dealt with all the weights in this case.

\subsection{Type $\mathbf{B_2}$} Here $2h-2 = 6$ so we need to deal with the cases $p=2$, $p=3$ and $p=5$.

\subsubsection{$\mathbf{p=2}$} After applying Theorem \ref{smallp}, the only weight left is $(1,1) = (p-1)\rho$ which is simple, so this case is done.

\subsubsection{$\mathbf{p=3}$} In this case we see that if $\Ext_G^1(k,L(\sigma))\neq 0$ then $\langle\sigma,\alpha_0^{\vee}\rangle\geq 3$ (by using the Jantzen sum formula to check that all the weights $(1,0)$, $(0,1)$ and $(0,2)$ are simple).

Thus by Proposition \ref{smallrank} we only need to consider weights $\lambda$ with $2\langle\lambda,\alpha_0^{\vee}\rangle \geq 3\cdot 3 = 9$, ie with $\langle\lambda,\alpha_0^{\vee}\rangle \geq 5$. This means we just need to consider the weights $(2,1)$ and $(2,2) = (p-1)\rho$. The latter is simple, as is $(2,1)$ (by applying the Jantzen sum formula), so this case is done.

\subsubsection{$\mathbf{p=5}$} After applying Theorem \ref{lambdaalpha0} we are left with the weights $(3,4)$, $(4,2)$, $(4,3)$ and $(4,4) = (p-1)\rho$. The last of these is simple, and so are $(4,2)$ and $(4,3)$ (seen by applying the Jantzen sum formula). This leaves us with just the weight $(3,4)$.

By Proposition \ref{summary}, if $\Ext_G^1(k,L(\sigma))\neq 0$ then $\langle\sigma,\alpha_0^{\vee}\rangle \geq 4$. Using the Jantzen sum formula, we see that there is a short exact sequence 
$$0\to L(3,4)\to \nabla(3,4)\to L(0,4)\to 0,$$ 
so we are done by Proposition \ref{smallrank} since it is not the case that $14 = \langle(3,4)+(0,4),\alpha_0^{\vee}\rangle \geq 5\cdot 4 = 20$.

\subsection{Type $\mathbf{G_2}$} Here $2h-2 = 10$ so we need to deal with the cases $p=2$, $p=3$, $p=5$ and $p=7$. However, we will not be able to deal with the case $p=7$.

\subsubsection{$\mathbf{p=2}$} Here the weights we need to consider after applying Theorem \ref{smallp} are $(0,1)$ and $(1,1) = (p-1)\rho$. The last of these is simple, and so is $(0,1)$ (as can be seen from the table on p. 90 of \cite{speciale}). So we are done in this case.

\subsubsection{$\mathbf{p=3}$} From the table on p. 85 of \cite{speciale} we see that if $\Ext_G^1(k,L(\sigma))\neq 0$ then $\langle\sigma,\alpha_0^{\vee}\rangle \geq 5$.
By Proposition \ref{smallrank} we thus only need to consider weights $\lambda$ with $2\langle\lambda,\alpha_0^{\vee}\rangle \geq 3\cdot 5 = 15$, ie with $\langle\lambda,\alpha_0^{\vee}\rangle\geq 8$. But the only restricted weight satisfying this are $(1,2)$ and $(2,2) = (p-1)\rho$. Since the latter of these is simple, we are left with $(1,2)$.

Looking at the same table again, we see that if $L(\mu)$ is a composition factor of $\nabla(1,2)$ then $\langle\mu,\alpha_0^{\vee}\rangle \leq 6$, so by Proposition \ref{smallrank}we are done, since $\langle(1,2),\alpha_0^{\vee}\rangle + 6 = 14 < 3\cdot 5 = 15$.

\subsubsection{$\mathbf{p=5}$} From the table on p. 83 of \cite{speciale} we see that if $\Ext_G^1(k,L(\sigma))\neq 0$ then $\langle\sigma,\alpha_0^{\vee}\rangle \geq 15$.
By Proposition \ref{smallrank} we thus see that we only need to consider weights $\lambda$ with $2\langle\lambda,\alpha_0^{\vee}\rangle \geq 5\cdot 15 = 75$. But there are no restricted weights satisfying this, so we are done.

\subsubsection{$\mathbf{p=7}$} As mentioned in the beginning of this section, we will not be able to deal with this case.

The reason for this is that there are in fact weights $\sigma,\lambda,\mu$ with $\lambda\in X_1$, $\mu\neq \lambda$ and such that we have  $\Ext_G^1(k,L(\sigma))\neq 0$, 
$[\nabla(\lambda):L(\mu)]\neq 0$ and $p\sigma\leq \lambda + \mu$. In particular, we can take $\sigma = 2\alpha_0$, which is sufficiently ``small'' that for many choices of $\lambda\in X_1$ 
there will be a $\mu$ satisfying the above. For more information about composition factors of the induced modules in this case, one can consult \cite{mertens85}. Thus, it no longer suffices to apply Proposition \ref{smallrank} in this case.

\section{Cohomological Criteria: Donkin's Conjecture} 

\subsection{Setting} Let $\lambda\in X_{+}$ and write $\lambda=\lambda_{0}+p^{r}\lambda_{1}$ where $\lambda_{0}\in X_{r}$ and $\lambda_{1}\in X_{+}$. Set 
$\hat{\lambda}_{0}=2(p^{r}-1)\rho-w_{0}(\lambda_{0})$, and $\Delta^{(p,r)}(\lambda)=T(\hat{\lambda}_{0})\otimes \Delta(\lambda_{1})^{( r )}$. It should be noted that the modules 
$\Delta^p(\lambda)$ in \cite{parshallscottpfilt} are not the same as our $\Delta^{(p,1)}(\lambda)$.

\begin{theorem} \label{T:nabla-Deltacompare1} Let $\lambda,\mu \in X_{+}$. Assume that 
$\operatorname{Hom}_{G_{r}}(T(\hat{\lambda}_{0}),L(\mu_{0}))^{(-r)}$ has a good filtration. 
Then 
\begin{equation*} 
\operatorname{Ext}_{G}^{n}(\Delta^{(p,r)}(\lambda),\nabla^{(p,r)}(\mu))=0 
\end{equation*} 
for $n>0$. 
\end{theorem} 

\begin{proof} We first apply the LHS spectral sequence: 
$$E_{2}^{i,j}=\text{Ext}^{i}_{G/G_{r}}(\Delta(\lambda_{1})^{( r )},\text{Ext}^{j}_{G_{r}}(T(\hat{\lambda}_{0}),L(\mu_{0}))\otimes \nabla(\mu_{1})^{( r )}) 
\Rightarrow \text{Ext}^{i+j}_{G}(\Delta^{(p,r)}(\lambda), \nabla^{(p,r)}(\mu)).$$
Since $T(\widehat{\lambda}_{0})$ is projective over $G_{r}$ this spectral sequence collapses and yields for $n\geq 0$ 
\begin{equation} 
\text{Ext}^{n}_{G}(\Delta^{(p,r)}(\lambda), \nabla^{(p,r)}(\mu))
\cong 
\text{Ext}^{n}_{G/G_{r}}(\Delta(\lambda_{1})^{( r )},\text{Hom}_{G_{r}}(T(\widehat{\lambda}_{0}),L(\mu_{0})) \otimes \nabla(\mu_{1})^{( r )}).
\end{equation}  
The result now follows by using the assumption that $\operatorname{Hom}_{G_{r}}(T(\hat{\lambda}_{0}),L(\mu_{0}))^{(-r)}$ has a good filtration and the 
fact that tensor products of induced modules admit a good filtration. 
\end{proof} 

In the case when $p\geq 2h-2$ (or in the cases when Donkin's Tilting Conjecture holds), one has $T(\hat{\lambda}_{0})\cong Q_{r}(\lambda_{0})$ for all $\lambda_{0}\in X_{r}$, and 
we can refine the aforementioned theorem. 

\begin{corollary} \label{C:nabla-Deltacompare2} Let $p\geq 2h-2$ and $\lambda,\mu \in X_{+}$. Then 
\begin{equation*} 
\operatorname{Ext}_{G}^{n}(\Delta^{(p,r)}(\lambda),\nabla^{(p,r)}(\mu))\cong 
\begin{cases} 0 & \text{$n>0$}\mbox{ or } \text{$\lambda\neq \mu$} \\
k & \text{$n=0$ and $\lambda=\mu$} 
\end{cases} 
\end{equation*} 
\end{corollary} 

\begin{proof} In this case $T(\hat{\lambda}_{0})\cong Q_{r}(\lambda_{0})$ for all $\lambda_{0}\in X_{r}$, and 
$\operatorname{Hom}_{G_{r}}(Q_{r}(\lambda_{0}),L(\mu_{0}))^{(-r)}$ has a good filtration because it is either 
$k$ or zero. Moreover, by \cite[Proposition II.4.13]{rags},
\begin{eqnarray}  
\text{Ext}^{n}_{G/G_{r}}(\Delta(\lambda_{1})^{( r )}, \nabla(\mu_{1})^{( r )} )
&\cong & \text{Ext}^{n}_{G}(\Delta(\lambda), \nabla(\mu)) \\
&\cong & \begin{cases} 0 & \text{$n>0$}, \text{$n=0$ and $\lambda_{1}\neq \mu_{1}$} \\
k & \text{$n=0$ and $\lambda_{1}=\mu_{1}$}. 
\end{cases} 
\end{eqnarray} 
The result now follows by applying these facts. 
\end{proof} 

If $M$ is a finite-dimensional rational $G$-module admitting a good $(p,r)$-filtration, then we denote by $[M:\nabla^{(p,r)}(\lambda)]_{(p,r)}$ the number of times $\nabla^{(p,r)}(\lambda)$ occurs in such a filtration for some $\lambda\in X_+$ (this is easily seen to be well-defined since each $\nabla^{(p.r)}(\lambda)$ has a unique highest weight).

We then have the following.

\begin{corollary}
Assume that $p\geq 2h-2$, let $M$ be a finite-dimensional rational $G$-module admitting a good $(p,r)$-filtration and $\lambda\in X_+$. Then $$[M:\nabla^{(p,r)}(\lambda)]_{(p,r)} = \dim\Hom_G(\Delta^{(p,r)}(\lambda),M).$$
\end{corollary}

\begin{proof}
This follows immediately from Corollary \ref{C:nabla-Deltacompare2}.
\end{proof}

\subsection{} We can now provide 
necessary and sufficient cohomological conditions for $\text{St}_{r}\otimes M$ to admit a good filtration.

\begin{lemma}\label{simplesummand}
Let $\lambda\in X_+$ and assume that $L(\lambda) = \nabla(\lambda)$. Let $M$ be a rational $G$-module with $M_{\lambda}\neq 0$ and such that $M_{\mu}\neq 0\implies \mu\leq\lambda$. Then $L(\lambda)$ is a direct summand of $M$.
\end{lemma}

\begin{proof}
Since $\lambda$ is largest among the weights of $M$, there is a surjective homomorphism $M\to \lambda$ as $B$-modules. By Frobenius reciprocity, we obtain a non-zero homomorphism $M\to \nabla(\lambda) = L(\lambda)$ which must therefore be surjective, 
which yields a short exact sequence $0\to N\to M\to L(\lambda)\to 0$.

We now just need to show that this sequence splits, and for this it is sufficient to show that $\Ext^1_G(L(\lambda),L(\mu)) = 0$ for all composition factors $L(\mu)$ of $N$. But since $\mu\leq\lambda$, if $\Ext_G^1(L(\lambda),L(\mu))\neq 0$ then $\mu\neq \lambda$ and $L(\mu)$ is a composition factor of $\nabla(\lambda)$ (by \cite[II.2.14]{rags}), which cannot be the case as $\nabla(\lambda)$ was assumed to be simple.
\end{proof}

The following lemma and its proof are essentially identical to \cite[Theorem 2.5]{donkin93} (that statement and proof are originally due to C. Pillen \cite[Corollary A]{pillen93}), except we have replaced 
$L((p^r - 1)\rho + w_0(\lambda))$ by $\Delta((p^r - 1)\rho + w_0(\lambda))$ in the statement and made minor changes in the proof in order to accomodate this change.

\begin{lemma}\label{directsummand}
Let $\lambda\in X_r$. Then $T(\hat{\lambda})$ is a direct summand of $\St_r\otimes \Delta((p^r - 1)\rho + w_0(\lambda))$.
\end{lemma}

\begin{proof}
First note that since $\St_r = L((p^r - 1)\rho) = \nabla((p^r - 1)\rho)$, we deduce from Lemma \ref{simplesummand} that $\St_r$ is a direct summand of both $\Delta(-w_0(\lambda)) \otimes \Delta((p^r - 1)\rho + w_0(\lambda))$ and $T((p^r - 1)\rho + w_0(\lambda))\otimes \Delta(-w_0(\lambda))$. So one can write 
$$\Delta(-w_0(\lambda))\otimes \Delta((p^r - 1)\rho + w_0(\lambda)) = \St_r\oplus V$$ 
where all weights of $V$ are strictly smaller than $(p^r-1)\rho$.

Thus we get that $\St_r\otimes \Delta((p^r - 1)\rho + w_0(\lambda))$ is a direct summand of $$W = T((p^r - 1)\rho + w_0(\lambda))\otimes \Delta(-w_0(\lambda))\otimes\Delta((p^r - 1)\rho + w_0(\lambda))$$ and by the above we have 
$$W = T((p^r - 1)\rho + w_0(\lambda)) \otimes (\St_r\oplus V) = T((p^r - 1)\rho + w_0(\lambda))\otimes \St_r\oplus T((p^r - 1)\rho + w_0(\lambda))\otimes V$$

Let $T$ be an indecomposable summand of $\St_r\otimes \Delta((p^r - 1)\rho + w_0(\lambda))$ containing the $\hat{\lambda}$-weightspace. The claim is now that $T$ is in fact isomorphic to $T(\hat{\lambda})$, and it is thus sufficient to show that $T$ is tilting.

By the previous considerations, $T$ is a direct summand of either $T((p^r - 1)\rho + w_0(\lambda))\otimes \St_r$ or $T((p^r - 1)\rho + w_0(\lambda))\otimes V$. But, since $T$ has a non-zero $\hat{\lambda}$-weight space, it cannot be a direct summand of the latter, where all the weights are strictly smaller than $\hat{\lambda}$ (since all weights of $V$ are strictly smaller than $(p^r - 1)\rho$). Hence, $T$ is a direct summand of $T((p^r - 1)\rho + w_0(\lambda))\otimes \St_r$, which is tilting, which means that $T$ itself is tilting.
\end{proof}

The following equivalent conditions should be compared to those in \cite[Theorem 2.4]{andersen01}, where the focus is on $\Ext$-groups involving various $T(\lambda)$, instead of the $\Delta^{(p,r)}(\lambda)$ used below. This difference is what enables us to obtain a condition only involving $\Ext^1$-groups, rather than having to involve all higher $\Ext$-groups.

\begin{theorem} \label{T:cohom-criteria} Let $M$ be a $G$-module with $\dim \operatorname{Hom}_{G}(\Delta^{(p,r)}(\lambda),M)< \infty$ for 
all $\lambda \in X_{+}$. The following are equivalent.  
\begin{itemize} 
\item[(a)] $\operatorname{St}_{r}\otimes M$ has a good filtration. 
\item[(b)] $\Hom_{G_{r}}(T(\hat{\lambda}_{0}),M)^{(-r)}$ has a good filtration for all $\lambda_{0}\in X_{r}$. 
\item[(c)] $\Ext_G^n(\Delta^{(p,r)}(\lambda),M) = 0$ for all $\lambda\in X_+$.  $n\geq 1$.
\item[(d)] $\Ext_G^1(\Delta^{(p,r)}(\lambda),M) = 0$ for all $\lambda\in X_+$.
\end{itemize} 
\end{theorem}

\begin{proof} One can use the same argument as in Theorem~\ref{T:nabla-Deltacompare1} to deduce that 
\begin{equation} 
\Ext_G^n(\Delta^{(p,r)}(\lambda),M)\cong \text{Ext}_{G/G_{r}}^{n}(\Delta(\lambda_{1})^{(r)}, \text{Hom}_{G_{r}}(T(\hat{\lambda}_{0}),M))
\end{equation} 
for all $\lambda=\lambda_{0}+p^{r}\lambda_{1} \in X_+$, $n\geq 0$. 
The equivalence of (b), (c), and (d) now follows by \cite[Proposition II.4.16]{rags}.  We now will show that (a) if and only if (c). 

Suppose that (a) holds. Note that $T(\hat{\lambda}_0)$ is a direct summand of $\St_r\otimes\Delta((p^r - 1)\rho + w_0(\lambda_0))$ for any $\lambda_0\in X_r$ by Lemma \ref{directsummand}. Hence, for any $\lambda = \lambda_0 + p^r\lambda_1\in X_+$ with $\lambda_0\in X_r$ we get that $\Delta^{(p,r)}(\lambda) = T(\hat{\lambda}_0)\otimes \Delta(\lambda_1)^{(r)}$ is a direct summand of $\St_r\otimes\Delta((p^r - 1)\rho + w_0(\lambda_0))\otimes \Delta(\lambda_1)^{(r)}$, and it 
is therefore sufficient to show that $\Ext_G^n(\St_r\otimes \Delta(\mu_0)\otimes\Delta(\mu_1)^{(r)},M) = 0$ for any $\mu_0\in X_r$, $\mu_1\in X_+$and $n\geq 1$. 
Since $\St_r\otimes M$ has a good filtration, so does $\St_r\otimes M\otimes \nabla(\mu)^{(r)}$ for any $\mu\in X_+$ by 
Proposition~\ref{reductiontosimple}, and thus we get $$\Ext_G^n(\St_r\otimes \Delta(\mu_0)\otimes \Delta(\mu_1)^{(r)},M)
\cong \Ext_G^n(\Delta(\mu_0),\St_r\otimes M\otimes \nabla(-w_0(\mu_1))^{(r)}) = 0$$ which proves that (c) holds.

On the other hand, assume that (c) holds. We will employ Ringel's criterion for the existence of good filtrations, that is $N$ has a good filtration if and only if $\Ext_G^n(T(\lambda),N) = 0$ for all dominant $\lambda$ and all $n\geq 1$ (for a proof, see \cite[Theorem 2.2]{andersen01}). So it suffices to show that $\Ext_G^n(\text{St}_r \otimes T(\lambda),M) = 0$ for all $\lambda\in X_{+}$, $n\geq 1$. 
Observe that $\text{St}_r \otimes T(\lambda)$ is a direct sum of $T(\mu)$'s with each $\mu$ of the form $(p^r - 1)\rho + \sigma$ for dominant $\sigma$ 
(since these summands are injective as $G_r$-modules), and if we write $\sigma=\sigma_0 + p^r\sigma_1$ then $T(\mu)$ is a direct summand of $T((p^r - 1)\rho + \sigma_0) \otimes T(\sigma_1)^{( r)}$ by \cite[Lemma II.E.9]{rags}. 
The tilting module $T(\sigma_1)$ has a Weyl-filtration, thus $T(\mu)$ is a direct summand of a module which has a filtration with factors of the form $T(\hat{\nu}) \otimes \Delta(\gamma)^{(r)}$ where 
$\nu=(p^r - 1)\rho + w_0(\sigma_{0})\in X_r$. Since $\Ext_G^n(T(\hat{\nu}) \otimes \Delta(\gamma)^{(r)},M) = 0$ by assumption for all 
restricted $\nu$ and all dominant $\gamma$, this shows that $\text{Ext}_G^n(T(\mu),M) = 0$ for $n\geq 1$. 
Hence, $\Ext_G^n(T(\lambda),\text{St}_r \otimes M) = 0$, so $\text{St}_r \otimes M$ has a good filtration.

\end{proof} 

Observe that if Conjecture~\ref{donkinconj} holds then the conditions in Theorem~\ref{T:cohom-criteria} would be equivalent to a rational $G$-module $M$ admitting a good 
$(p,r)$-filtration. 

\subsection{}We now present a surprising consequence of the aforementioned theorem. 

\begin{corollary} Let $M$ be a $G$-module which admits good filtration. Then 
$\Hom_{G_{r}}(T(\hat{\lambda}_{0}),M)^{(-r)}$ has a good filtration for all $\lambda_{0}\in X_{r}$.  
\end{corollary} 

\begin{proof}. If $M$ admits a good filtration then $\text{St}_{r}\otimes M$ admits a good filtration. The result follows from Theorem~\ref{T:cohom-criteria}. 
\end{proof} 

It is interesting to note that this results does not hold in general for arbitrary tilting modules. For example, for $T(0)\cong k$, van der Kallen \cite{vanderkallen93} produced an example of a rational $G$-module $M$
admitting a good filtration such that $\text{Hom}_{G_{r}}(k,M)^{(-r)}$ does not admit a good filtration. 

\subsection{Donkin's $(p,r)$-Conjecture, redux} We can now prove that the verification of Donkin's Tilting Module Conjecture 
guarantees that the ``if'' direction of Donkin's $(p,r)$-Filtration Conjecture holds over fields of arbitrary characteristic. 

\begin{theorem} \label{T:DonkinConjecture2(-->)} Assume that any one of the following conditions holds for $G$. 
\begin{itemize} 
\item[(a)] $T(\hat{\delta})=Q_{r}(\delta)$ for all $\delta\in X_{r}$ 
\item[(b)] $\Hom_{G_{r}}(T(\hat{\delta}),L(\tau))$ has a good filtration for $\delta,\tau \in X_{r}$ 
\item[(c)] $p\geq 2(h-1)$. 
\end{itemize} 
If $M$ has a good $(p,r)$-filtration then $\St_r\otimes M$ has a good filtration. 
\end{theorem} 

\begin{proof} Any one of the conditions insures that $\Hom_{G_{r}}(T(\hat{\delta}),L(\tau))$ has a good filtration for $\delta,\tau \in X_{r}$. 
Therefore, $\text{Ext}^{n}_{G}(\Delta^{(p,r)}(\lambda),\nabla^{(p,r)}(\mu)) = 0$ for $\lambda,\mu\in X_{+}$, $n\geq 1$ by Theorem~\ref{T:nabla-Deltacompare1}. 
It follows that if $M$ has a good $(p,r)$-filtration then $\text{Ext}^{n}_{G}(\Delta^{(p,r)}(\lambda),M) = 0$ for $\lambda\in X_{+}$, $n\geq 1$. Consequently, 
$\text{St}_{r}\otimes M$ has a good filtration by Theorem~\ref{T:cohom-criteria}.
\end{proof}

%%%%%%%%%
%%References
%%%%%%%%%

\providecommand{\bysame}{\leavevmode\hbox
to3em{\hrulefill}\thinspace}

\end{document}